\documentclass[10pt,a4paper,dk,reqno]{amsart}
\let\amsmarkboth\markboth    
\usepackage{amsmath,amsfonts,amssymb,amsthm,amscd}
\usepackage[english]{babel}
\usepackage[mathscr]{eucal}
\usepackage{graphicx}
\usepackage{appendix}

\makeatletter
\let\markboth\amsmarkboth
\bbl@redefine#1#2{%
   \def\bbl@arg{#1}%
   \ifx\bbl@arg\@empty
     \toks@{}%
   \else
     \toks@{\noexpand\foreignlanguage{%
              \languagename}{%
              \noexpand\bbl@restore@actives#1}}%
   \fi
   \def\bbl@arg{#2}%
   \ifx\bbl@arg\@empty
     \toks8{}%
   \else
     \toks8{\noexpand\foreignlanguage{%
              \languagename}{%
              \noexpand\bbl@restore@actives#2}}%
   \fi
   \edef\bbl@tempa{\the\toks@}%
   \edef\bbl@tempb{\the\toks8}%
   \protected@edef\bbl@tempa{%
     \noexpand\org@markboth{\bbl@tempa}{\bbl@tempb}}%
   \bbl@tempa
}
\DeclareRobustCommand*\ams@disablelinebreak{\def\\{ \ignorespaces}}
\def\maketitle{\par
   \@topnum\z@ %
   \@setcopyright
   \thispagestyle{firstpage}%
   \uppercasenonmath\shorttitle
   \ifx\@empty\shortauthors \let\shortauthors\shorttitle
   \else \andify\shortauthors
   \fi
   \@maketitle@hook
   \begingroup
   \@maketitle
   \toks@\@xp{\shortauthors}\@temptokena\@xp{\shorttitle}%
   \protected@edef\@tempa{%
     \@nx\markboth{\ams@disablelinebreak
       \@nx\MakeUppercase{\the\toks@}}{\the\@temptokena}}%
   \@tempa
   \endgroup
   \c@footnote\z@
   \@cleartopmattertags
}

\makeatother

\numberwithin{equation}{section}

\DeclareMathOperator{\im}{{\Im\mathrm{m}}\,} \DeclareMathOperator{\re}{{
\Re\mathrm{e}}\,}

\newtheorem{theorem}{Theorem}[section]
\newtheorem{proposition}[theorem]{Proposition}
\newtheorem{lemma}[theorem]{Lemma}
\newtheorem{corollary}[theorem]{Corollary}

\newtheorem{remark}[]{Remark}
\theoremstyle{definition}

\renewcommand{\P}{\Psi}
\newcommand{\dt}{\partial_t}
\newcommand{\C}{\mathbb{C}}
\newcommand{\R}{\mathbb{R}}

\newcommand{\eps}{\varepsilon}

\title[]{Global existence and collisions\\ for symmetric configurations of nearly parallel vortex filaments}

\thanks{V.B. is partially
  supported by the ANR project ``R.A.S.''.} 

\author[]{Valeria Banica}
\email{Valeria.Banica@univ-evry.fr}
\date{\today}

\author[]{Evelyne Miot}
\email{Evelyne.Miot@math.u-psud.fr}

\begin{document}
\maketitle
\begin{abstract}We consider the Schr\"odinger system 
with Newton-type interactions that was derived by R.~Klein, A.~Majda and K.~Damodaran \cite{KMD} to modelize 
the dynamics of $N$ nearly parallel vortex filaments in a 3-dimensional homogeneous incompressible fluid. 
The known large time existence results are due to C.~Kenig, G.~Ponce and L.~Vega \cite{KPV} and concern the interaction of two filaments and 
particular configurations of three filaments.  In this article we prove large time existence results for particular configurations of four nearly parallel filaments and for a class of configurations of $N$ nearly parallel filaments for any $N\geq 2$. We also show the existence of 
travelling wave type dynamics. Finally we describe configurations leading to collision.

\end{abstract}

 \section{Introduction}

In this paper we study the dynamics of $N$
interacting vortex filaments in a 3-dimensional homogeneous incompressible fluid. 
We focus on filaments that are all nearly parallel to 
the $z$-axis. They are described by means of complex-valued functions 
$\Psi_j(t,\sigma)\in \C$, $1\leq j\leq N$, where 
$t\in \R$ is the time, $\sigma\in \R$ parameterizes the $z$-axis, and $\Psi_j(t,\sigma)$ is the position of the 
$j$-th filament. 
A simplified model for the dynamics of such nearly parallel filaments has been derived by 
R.~Klein, A.~Majda and K.~Damodaran \cite{KMD} in the form of the following 1-dimensional Schr\"odinger system of equations
\begin{equation}
 \label{syst:interaction}
\begin{cases}
\displaystyle 
i\dt \Psi_j+ \Gamma_j \partial_\sigma^2 \Psi_j +\sum_{k\neq j} 
\Gamma_k \frac{\Psi_j-\Psi_k}{|\Psi_j-\Psi_k|^2}=0,\quad1\leq j\leq N,\\
\Psi_j(0,\sigma)=\Psi_{j,0}(\sigma). 
\end{cases}
\end{equation}
Here  $\Gamma_j$ is a real number representing the circulation of the j-th filament\footnote{The free Schr\"odinger operator derived in \cite{KMD} is actually 
$i\partial_t +\alpha_j \Gamma_j \partial_\sigma^2$, where $\alpha_j$ is 
another vortex core parameter related to the $j$-th filament. For 
simplicity we assume throughout the paper that $\alpha_j=1$.}.   
In the case where $\Psi_j(t,\sigma)=\Psi_j(t)=X_j(t)$ are exactly parallel filaments, 
Syst. \eqref{syst:interaction} reduces to 
the well-known point vortex system arising in 2-dimensional homogeneous incompressible fluids
\begin{equation}
 \label{syst:point-vortex}
\begin{cases}
\displaystyle 
i\frac{d X_j}{dt}+\sum_{k\neq j} \Gamma_k \frac{X_j-X_k}{|X_j-X_k|^2}=0,\quad 1\leq j\leq N,\\
X_j(0)=X_{j,0}. 
\end{cases}
\end{equation}
The Syst. \eqref{syst:interaction} combines on the one hand the linearized self-induction 
approximation for each vortex filament, given by the linear Schr\"odinger equation, and on the other 
hand the interaction of the filaments, for any $\sigma$, by the point vortex system. 
Solutions of the simplified model \eqref{syst:interaction} have remarkable 
mathematical and physical properties, as described in \cite{MaBe}. 
The main issue in this context is 
the possibility of collision of at least two of the filaments in finite time at some point $\sigma$.  

Before presenting the known results on nearly parallel vortex filaments let us briefly review some classical
facts on the point vortex system \eqref{syst:point-vortex}. Its dynamics preserves the center of inertia
$\sum_j \Gamma_j X_j(t)$, the angular momentum $\sum_j \Gamma_j |X_j(t)|^2$  and the quantities 
$$\sum_{j\neq k} \Gamma_j \Gamma_k  \ln \left|X_j(t)-X_k(t)\right|^2\,\,,\,\,
\sum_{j\neq k}\Gamma_j \Gamma_k \left|X_j(t)-X_k(t)\right|^2.$$In case of circulations having all the 
same signs this implies that no collision among the vortices can occur in finite time. Therefore 
there exists a unique global $\mathcal C^1$ solution $\left(X_j(t)\right)_j$ to \eqref{syst:point-vortex}. 
For $N=2$ global existence still holds 
independently of the circulation signs since $|X_1(t)-X_2(t)|$ remains constant. 
When dealing with more than two vortices the single-sign assumption of the circulations 
really matters - explicit examples of configurations leading to collapse in finite time 
have been given by self-similar shrinking triangles \cite{Aref79}. For any circulations  the equilateral triangle is a rotating or translating configuration, 
and for identical circulations the ends and the middle of a segment form 
also a relative equilibrium configuration. For $N\geq 4$ and identical circulations $\Gamma_j=\Gamma\,\forall j$, 
vertices of regular polygons also form 
relative equilibrium configurations. They rotate around the center of inertia  
 with constant angular velocity $\omega=\Gamma(N-1)/(2R^2)$, where $R$
is the size of the polygon. These polygon configurations are stable if 
and only if $N\leq 7$. The proof of this result, conjectured by Kelvin in 1878, 
was recently completed by L.~G.~Kurakin and V.~I.~Yudovitch in 2002 \cite{KuYu} (see also \cite{No}). Finally, the configuration formed by adding to an $N$-polygon configuration
 one point of arbitrary circulation $\Gamma_0$ 
at the center of inertia, is an relative equilibria rotating with constant angular velocity 
$\omega=[\Gamma(N-1)+2\Gamma_0]/(2R^2)$. A 
natural observation to be done is that as $N$ increases the dynamics gets more and more sophisticated. 

A first result on nearly parallel vortex filaments has been given in \cite{KMD}. 
The authors proved that for $N=2$ the linearized system 
around the exactly parallel filaments solution of \eqref{syst:point-vortex} is stable if the circulations have the same sign and unstable otherwise. 
Moreover, they made numerical simulations suggesting 
global existence for \eqref{syst:interaction} in the first case and collision in finite time in the second case. 
Their first conjecture on global existence was proved then by C.~Kenig, G.~Ponce and L.~Vega \cite{KPV} for filaments $\Psi_j$
obtained as small $H^1$ perturbations of exactly 
parallel filaments $X_j$,
\begin{equation}
\label{hyp:pert}
 \Psi_j(t,\sigma)=X_j(t)+u_j(t,\sigma),\quad 1\leq j\leq N.
\end{equation}
More precisely, it has been proved in \cite{KPV} that for $u_j(0)$ sufficiently 
small in $H^1(\mathbb R)$ -and therefore in $L^\infty(\mathbb R)$- global existence 
and uniqueness of the solution to Syst. \eqref{syst:interaction} hold for all vortex solutions $(X_j)_j$
of equal circulations and such that $|X_j(t)-X_k(t)|=d $ for $1\leq j\neq k\leq N$.  The only such possible 
configurations are $N=2$ with any pair $(X_1,X_2)$, and $N=3$ with $(X_1,X_2,X_3)$ an equilateral triangle.
Moreover, local existence und uniqueness hold for any number $N$ of filaments and any circulations $\Gamma_j$ and 
the solution exists at least up to times of order $ |\ln \sum_j\|u_j(0)\|_{H^1}|$.

Finally, let us mention that P.-L.~Lions and A.~Majda \cite{LM} developed an equilibrium statistical theory for nearly parallel filaments using the approximation given by Syst. \eqref{syst:interaction}.

\medskip

The purpose of this article is to study other specific configurations of 
vortex filaments. In order to obtain large time existence results  we will strongly use the symmetry properties 
of the configuration of the straight filaments $(X_j)_j$ in itself, and  
those of the perturbation $(u_j)_j$ on the other hand.

\medskip
In the first part of this paper we focus on the case where $N\geq 3$ and $(X_j)_j$ is a 
regular rotating polygon of radius $1$ with $N$ vertices, with or without its center.  The
index $j=0$ refers to the center of the polygon and $1\leq j\leq N$ to the vertices of the polygon. Since 
\eqref{syst:interaction} is invariant under translations, we can suppose that the center 
of inertia of the polygon is set at the origin, i.e. $X_0(t)=0$ for all $t$. We shall impose that the circulations in the vertices have the same value $\Gamma$ and that $\omega$ has the same sign as $\Gamma$. For simplicity we consider 
$$\Gamma_j=1,\quad 1\leq j\leq N.$$ 
In the cases where the center of the polygon is not considered, the angular speed $\omega$ is $(N-1)/2$, hence positive. In the cases when the center of the polygon is considered, the circulation $\Gamma_0$ must be larger than $-(N-1)/2$.

We will consider very specific perturbations of the configuration $(X_j)_j$,
assuming that all the perturbations are the same for each of the straight filaments, a dilation 
combined with a rotation. More precisely we shall focus on solutions having the form
\begin{equation}\label{hyp:pert-N}
 \Psi_j(t,\sigma)=X_j(t)\Phi(t,\sigma),
\end{equation}
with $\Psi(t,\sigma)$ close to $X_j(t)$ in some sense as $|\sigma|\rightarrow\infty$. Let 
us notice that this dilation-rotation type of perturbations keeps the symmetry of the polygon 
for all $(t,\sigma)$. A natural example of such perturbations are the ones 
with $\Phi-1$ small in $H^1(\R)$.  Our result below allows to handle a larger class of 
perturbations of the 
regular rotating polygon, including also for example all small constant rotations of the polygon.

\begin{theorem}
 \label{thm:main-2}Let $N\geq 3$ and $(X_j)_j$ be the equilibrium solution 
given by a regular rotating polygon of radius $1$, with or without its center, with $\Gamma_j=1$ for $1\leq j\leq N$ and positive angular velocity $\omega$.
Assume that 
$$\Psi_{j,0}(\sigma)=X_{j,0}\Phi_0(\sigma),$$ with $\Phi_0$ such that 
\begin{equation*}
 \mathcal E(\Phi_0)= \frac{1}{2}\int |\partial_\sigma \Phi_0|^2
+\frac{\omega}{2}\int \left(|\Phi_0|^2-1-\ln |\Phi_0|^2\right)
\end{equation*}
satisfies $\mathcal E(\Phi_0)\leq \eta_1$, where $\eta_1$ is an absolute constant\footnote{introduced in Lemma \ref{lemma:ginzburg} 
below.}. 
Then 
there exists a unique global solution $(\Psi_j)_j$ of \eqref{syst:interaction}, with this initial datum, such that
$$\Psi_{j}(t,\sigma)=X_{j}(t)\Phi(t,\sigma),\quad t\in \R$$
  with
$\Phi-\Phi_0\in C\left(\R,H^1(\R)\right)$. Moreover
$$\frac 34
\leq \frac{|\Psi_j(t,\sigma)-\Psi_k(t,\sigma)|}{|X_j(t)-X_k(t)|}\leq \frac 54,\quad t,\sigma\in \R.$$ 
In particular, if 
$\Phi_0(\sigma)\overset{|\sigma|\rightarrow\infty}{\longrightarrow} 1$ then $\Psi_j(t,\sigma)\overset{|\sigma|\rightarrow\infty}{\longrightarrow} X_j(t)$ $\forall t$, and 
if $\Phi_0\in 1+H^1(\R)$ then $\Psi_j-X_j\in C\left(\R,H^1(\R)\right)$.

\end{theorem}
\begin{remark}
 Theorem \ref{thm:main-2} does not assert that if initially $\|\Phi_0-1\|_{H^1}$ is small
 then $\|\Phi(t)-1\|_{H^1}$ remains small for all $t$.
\end{remark}

Our analysis is based on the observation that 
the solution $(\Psi_j)_j$ to Syst. \eqref{syst:interaction}  satisfies \eqref{hyp:pert-N} if and only if  
$\Phi$ is solution to the equation
\begin{equation}\label{eq:BM}
 i\partial_t \Phi+\partial_\sigma^2 \Phi+\omega \frac {\Phi}{|\Phi|^2}(1-|\Phi|^2)=0.
\end{equation}Eq. \eqref{eq:BM} is an hamiltonian equation, which preserves the energy 
\begin{equation}\label{def:energy-BM}
 \mathcal E(\Phi)= \frac{1}{2}\int |\partial_\sigma \Phi|^2+\frac{\omega}{2}\int \left(|\Phi|^2-1-\ln |\Phi|^2\right).
\end{equation}
Note that in the setting of Theorem \ref{thm:main-2} the solutions satisfy $|\Phi|\simeq 1$, so that  
Eq. \eqref{eq:BM}  is formally similar to the well-known Gross-Pitaevskii 
equation
\begin{equation}\label{eq:GP}
 i\partial_t \Phi+\partial_\sigma^2 \Phi+\omega \Phi(1-|\Phi|^2)=0,
\end{equation}
with energy  given by 
\begin{equation*}
 \mathcal E_{GP}(\Phi)= \frac{1}{2}\int |\partial_\sigma \Phi|^2+\frac{\omega}{4}
\int \left(|\Phi|^2-1\right)^2. 
\end{equation*}
In fact we shall see that both functionals $\mathcal{E}(\Phi)$ and $\mathcal{E}_{GP}(\Phi)$ are comparable  
whenever $|\Phi|\simeq 1$. A key point in the  
proof is, as in \cite{KPV}, the fact that if $\mathcal E(\Phi_0)$ is small then the solution $\Phi$ enjoys the property
\begin{equation}\label{ineq:coercivity}
  \sup_{t\in \R}\left\| |\Phi(t)|^2-1\right\|_{L^\infty}\leq \frac{1}{4}.\end{equation} 
This allows us to establish Theorem \ref{thm:main-2} by using the techniques introduced in \cite{Zh2} by P.~E.~Zhidkov (see also P.~G\'erard   \cite{PG0}, \cite{PG}) for solving the Gross-Pitaevskii equation in the energy space.\par
In the case where $\Phi_0\in 1+H^1(\R)$ we mention that the proof in \cite{KPV} can be adapted here, 
by showing that some quantities are still conserved even though $|X_j(t)-X_k(t)|$ are not all the 
same.

\medskip

As far as we have seen, global existence and uniqueness 
of the filaments hold for $N=2$ with any $(X_j)_j$ and any small pertubations, 
for $N=3$ with $(X_j)_j$ the equilateral triangle stable equilibrium and any small pertubations, for any $N\geq 2$ 
with $(X_j)_j$ the regular polygon equilibrium
and any small pertubations with strong symmetry conditions. We expect then that 
global existence might hold for small $N$ and less restrictive conditions on the perturbations.

\medskip

In the second part of this paper we study the case
\begin{equation*}
 N=4,\quad \Gamma_j=1,
\end{equation*}
and we assume that
$
 (X_j)_j=(X_1,X_2,X_3,X_4)
$
 is a square of 
radius $1$ rotating with constant angular speed. 
Again, since \eqref{syst:interaction} is 
invariant under translations, we can suppose that the square is centered at the origin. 
Our main result  in this case may be formulated as follows.
\begin{theorem}
 \label{thm:main}
Let $N=4$ and $(X_j)_j$ be the equilibrium solution given by a rotating square of radius $1$ with $\Gamma_j=1$. 
 Let $(u_{j,0})_j\in H^1(\R)^4$ and set $\Psi_{j,0}=X_{j,0}+u_{j,0}$.

We introduce the energy\footnote{Note that $\mathcal{E}_0\geq 0$.}
\begin{equation*}\begin{split}\mathcal{E}_0&=\frac{1}{2}\sum_{j}
  \int \left|\partial_\sigma \P_{j,0}(\sigma)\right|^2\,d\sigma\\
&+\frac{1}{2}\sum_{j\neq k}  \int -\ln\left(\frac{|\Psi_{j,0}(\sigma)-\Psi_{k,0}(\sigma)|^2}{|X_{j,0}-X_{k,0}|^2}\right)
+\left(\frac{|\Psi_{j,0}(\sigma)-\Psi_{k,0}(\sigma)|^2}{|X_{j,0}-X_{k,0}|^2}-1\right)\,d\sigma.\end{split}\end{equation*}

We also introduce the quantity
\begin{equation*}
 \tilde{\mathcal{E}_0}=\max\left\{\mathcal{E}_0;
\frac{\|u_{1,0}+u_{3,0}\|_{L^2}^2}{2}+\frac{\|u_{2,0}+u_{4,0}\|_{L^2}^2}{2}\right\}
\end{equation*}
and we assume that 
\begin{equation*}
 \tilde{\mathcal{E}_0}\leq \eta_2
\end{equation*}
for an absolute small constant $\eta_2>0$. Then there exists an absolute constant $C>0$, and there exists a time $T$,
with
$$T\geq C
\min\left\{\frac{1}{{\tilde{\mathcal{E}_0}}^{1/4}\max_{j,k}
\|u_{j,0}-u_{k,0}\|_{L^2}^{1/2}},\frac{1}{\tilde{{\mathcal{E}_0}}^{1/3}}\right\},$$
such that there exists 
a unique corresponding solution 
$(\Psi_j)_j$
to Syst. \eqref{syst:interaction} on $[0,T]$, satisfying 
 $\Psi_j=X_j+u_j$, with $u_j\in C\left([0,T],H^1(\R)\right)$, and such that
\begin{equation*}
\frac 34\leq \frac{|\Psi_j(t,\sigma)-\Psi_k(t,\sigma)|}{|X_j(t)-X_k(t)|}\leq \frac 54,\quad t\in [0,T],
\quad \sigma\in \R. 
\end{equation*}

Finally, if the initial perturbation is parallelogram-shaped, namely
\begin{equation*}
 \|u_{1,0}+u_{3,0}\|_{L^2}=\|u_{2,0}+u_{4,0}\|_{L^2} =0,
\end{equation*}
then the solution $(\Psi_j)_j$ is globally defined.

\end{theorem}

\begin{remark}In the proof of Theorem \ref{thm:main}
we shall actually establish 
a local existence result for any $N$, any parallel 
configuration $(X_j)_j$, any set of positive circulations $(\Gamma_j)_j$ and any perturbations with small energy,  but not necessarily small in $H^1$. This is a slight improvement of the result in \cite{KPV}, see also the next two remarks. 
\end{remark}

\begin{remark} 	As we shall see, we can infer from  
the smallness of the energy $\mathcal{E}_0$ and from Sobolev embeddings that the nearly parallel filaments $\Psi_{j,0}$
are not too far from the straight filaments $X_{j,0}$ and that
$ \mathcal {E}_0\leq C\sum_j \|u_{j,0}\|_{H^1}^2.$ Conversely, if we assume that $\sum_j \|u_{j,0}\|_{H^1}$ 
is sufficiently small then one can show that $ 
\tilde{\mathcal {E}_0}\leq C\sum_j \|u_{j,0}\|_{H^1}^2$ and
the assumptions of Theorem \ref{thm:main} are satisfied. Therefore the hypothesis on the energy is less restrictive than the one on the $H^1$ norm, see also the next remark.
\end{remark}

\begin{remark}From 
$0\leq \mathcal {E}_0\leq C\sum_j \|u_{j,0}\|_{H^1}^2$ it follows that 
 $\tilde{\mathcal {E}_0}\leq C\sum_j \|u_{j,0}\|_{H^1}^2$ so
the time of existence is a priori larger than in \cite{KPV}. Moreover, 
for all $\epsilon>0$ Theorem \ref{thm:main} allows for initial perturbations of the form
$$\Psi_{j,0}^\epsilon(\sigma)=e^{i\varphi^\epsilon(\sigma)} X_{j,0}+T^\epsilon(\sigma),$$
with $\varphi^\epsilon,T^\epsilon$ such that $\|(\varphi^\epsilon,T^\epsilon)\|_{H^1}
=O(1)$.  This amounts to rotating and translating the square $(X_j)_j$ at each level $\sigma$. By taking  
oscillating phases of the form 
$\varphi^\epsilon(\sigma)=\sqrt{\eps} \varphi_0(\eps \sigma)$ with a fixed $\varphi_0\in H^1$, which 
implies 
$\|\varphi^\epsilon\|_{L^2}\geq O(1)$, 
$\|\nabla\varphi^\epsilon\|_{L^2}=O(\epsilon)$ and by choosing $T^\epsilon$ such that $\|T^\epsilon\|_{ H^1}=O(\epsilon)$ 
we compute
\begin{equation*}
 \tilde{\mathcal E_0}=O(\epsilon^2),\quad \sum_j\|u_{j,0}\|_{H^1}^2\geq O(1).
\end{equation*}
Therefore Theorem \ref{thm:main}  provides a unique solution a least up to time of order $1/\sqrt{\eps}$, while the $H^1$ norm of the perturbations is of order one. 
This suggests that the energy space is more appropriate for the analysis of \eqref{syst:interaction} 
than classical Sobolev spaces.
\end{remark}

The proof of Theorem \ref{thm:main} follows the one of Theorem \ref{thm:main-2} combined with 
the one in \cite{KPV}. 
In particular, we consider, as in \cite{KPV}, the energy
\begin{equation}
\label{def:energy-losange}
\begin{split}\mathcal{E}(t)&=\frac{1}{2}\sum_{j}
  \int \left|\partial_\sigma \P_j(t,\sigma)\right|^2\,d\sigma\\
&+\frac{1}{2}\sum_{j\neq k}  \int -\ln\left(\frac{|\Psi_j(t,\sigma)-\Psi_k(t,\sigma)|^2}{|X_j(t)-X_k(t)|^2}\right)
+\left(\frac{|\Psi_j(t,\sigma)-\Psi_k(t,\sigma)|^2}{|X_j(t)-X_k(t)|^2}-1\right)\,d\sigma,\end{split}\end{equation} 
and show that the solution can be extended as long as $\mathcal E(t)$ remains small. For this purpose we 
show that $u_j$ can be extended locally from a time $t_0$ by a fixed point argument for small $H^1$ 
perturbations $w_j$ of the linear evolutions of the initial data, i.e. $u_j(t)=e^{i(t-t_0)\partial_\sigma^2}u_j(t_0)+w_j(t)$. In here we use crucially the fact that the deviation $e^{i(t-t_0)\partial_\sigma^2}u_j(t_0)-u_j(t_0)$ can be upper-bounded in $L^\infty$ in terms of the energy at the initial time $\mathcal E(t_0)$. 
As observed in  \cite{KPV}, for any two parallel filaments and for the equilateral triangle configuration the energy is conserved, i.e.
 $\mathcal{E}(t)=\mathcal{E}(0)=\mathcal{E}_0$, so that
global existence follows for small energy perturbations. Unfortunately, under the assumptions
of Theorem \ref{thm:main} the energy  is no longer conserved (unless the perturbation $(u_j)_j$ is parallelogram-shaped). 
Instead, we estimate its evolution in time, showing that it does not increase too fast, 
and this control enables us to obtain a large time of existence.

\medskip

We finally mention another collection of dynamics that is 
 governed by the linear Schr\"odiger equation. 
For shifted perturbations $\Psi_j=X_j+u$, for any 
$X_j$ with $\Gamma_j$ the same, we obtain that $u$ is a solution of the linear Schr\"odiger equation.  
So if $u$ is regular enough, it 
has constant $H^1$ norm, so the filaments remain separated for all time. Moreover, due to the 
dispersive inequality for the linear Schr\"odinger equation, the perturbations spread in time 
along the parallel configuration $X_j$. Finally, we get examples of $\mathcal C^\infty$ perturbations 
decaying at infinity  that generate a singularity in finite time by considering less regular 
pertubations than $H^1$ that lead to a $L^\infty$ dispersive blow-up for the linear Schr\"odinger. 
The self-similar linear Schr\"odinger solution constructed from  homogeneous data $|x|^{-p}$ with 
$0<p<1$ in \cite{CaWe} leads to solutions blowing-up in $L^\infty$ in finite time at one point. Also, the linear 
Schr\"odinger evolution of $e^{i|x|^2}/(1+|x|^2)^m$ with $1/2<m\leq 1$ has been 
proved in \cite{BoSa} to be an $L^2$ solution whose modulus blows-up in finite time at one point.

\medskip

The third part of this work is devoted to travelling waves for Syst. \eqref{syst:interaction}. Let us recall that in the case of one single filament, a travelling wave dynamics was exhibited by H.~Hasimoto \cite{Has} and experimentally observed by E. J.~Hopfinger and F.K.~Browand \cite{HB}. Here we construct travelling waves for several filaments via finite energy travelling wave 
solutions to Eq. \eqref{eq:BM}, i.e. solutions of the form 
$\Phi(t,\sigma)=v(\sigma+ct)$, with  $v$ solution of the equation
\begin{equation}
 \label{eq:TW}
ic v'+v''+\omega\frac{v}{|v|^2}(1-|v|^2)=0
\end{equation}
and having finite energy, 
\begin{equation}
 \mathcal{E}(v)=\frac{1}{2}\int|\partial_\sigma v|^2+\frac{\omega}{2}\int \big( |v|^2-1-\ln |v|^2\big)<\infty.
\end{equation} 
As in Theorem \ref{thm:main-2} we assume that $\omega>0$. In order to avoid having $v$ approaching zero we shall impose that the energy is small. \par
Existence, stability issues and qualitative behaviour near the speed of sound of travelling waves for Gross-Pitaevskii-type equations and related problems were extensively studied in the past years (see for instance \cite{DeB, BS, DiMGa, Gr, BeGrSaSm, Ma, ChRo} and the references therein). For the one-dimensional Gross-Pitaevskii equation \eqref{eq:GP}, finite energy travelling waves (referred to as "grey solitons"
in the context of non-linear optics) are known to exist for all 
$0<c<\sqrt{2\omega}$, and they have the explicit
form (see e.g. \cite{Gr}) 
\begin{equation*}
\begin{split}
 v(\sigma)=v_c(\sigma)&=\sqrt{ 1-\frac{\frac{1}{2\omega}(2\omega-c^2)}{\cosh^2\left ( \frac{\sqrt{2\omega-c^2}}{2} \sigma\right)}}
\,e^{i \arctan \frac{\omega e^{\sqrt{2\omega-c^2}\sigma}+c^2-\omega}{c\sqrt{2\omega-c^2}}-i\arctan
\frac{c}{\sqrt{2\omega-c^2}}}.
\end{split}
\end{equation*}
The modulus  $|v_c|$ of such maps is close to $1$ when $c$ is close to $\sqrt{2\omega}$, in which case
 $\mathcal{E}(v_c)\leq C\mathcal{E}_{GP}(v_c)\leq C(2\omega-c^2)^{3/2}$ (see \cite{Gr}), so the energy is finite and
 as small as needed. Note that therefore the maps $v_c$, with $c$ close to $\sqrt{2\omega}$ enter the class of
perturbations presented in Theorem \ref{thm:main-2}. Our next result in this context is the following.
\begin{theorem}
 \label{thm:TW}
Let $c$ such that $0<2\omega-c^2<\eta_3$ for an absolute small constant $\eta_3>0$. There exists a travelling wave solution to Syst. \eqref{syst:interaction} $$\Psi_j(t,\sigma)=e^{it\omega+i\frac{2\pi j}{N}}v(\sigma+ct),$$
where $v\in C^\infty(\R)$ is a solution to Eq.
\eqref{eq:TW}, with finite energy $
 \mathcal{E}(v)\leq C(2\omega-c^2)^{3/2}$, such that $v$ never vanishes. The modulus
 $|v|$ is an even function, increasing on $[0,\infty)$ and satisfying on $\R$
\begin{equation*}
 0<1-|v(\sigma)|^2 <\min\left\{\frac{3}{2\omega}(2\omega-c^2),\, C\sqrt{2\omega-c^2}e^{-\sqrt{2\omega-c^2}^{\,-}|\sigma|}\right\}.
\end{equation*} Finally,
we have a limit at infinity
\begin{equation*}
 v(\sigma)\to \exp(i\theta_{\pm}),\quad \sigma\to \pm \infty,
\quad \text{with }
 |\theta_+-\theta_-|\leq C\sqrt{2\omega-c^2}.
\end{equation*}

Here $C$ denotes an absolute numerical constant.

\end{theorem}

\medskip

It has been noticed in \cite{KPV} that the Galilean invariance of Syst. 
\eqref{syst:interaction} leads to helix-shaped vortex filaments.
In here, on one hand Eq. \eqref{eq:BM} is invariant under Galilean
  transform, i.e. $\Phi_{\nu}(t,\sigma)=e^{-it\nu^2+i\nu \sigma}\,\Phi(t,\sigma-2t\nu)$ is
  also a solution $\forall \nu\in\mathbb R$. On the other hand
  $X_j(t)=e^{it\omega+i\frac{2\pi j}{N}}$ for $j\neq 0$, so
 \begin{equation*}\begin{split}\Psi_{j,\nu}(t,\sigma)&=e^{it(\omega-\nu^2)+i\nu \sigma+ i\frac{2\pi
j}{N}}\,\Phi(t,\sigma-2t\nu)\\
  &=e^{it(\omega-\nu^2)+i\nu \sigma+ i\frac{2\pi j}{N}}\,v(\sigma+t(c-2\nu)).\end{split}\end{equation*}
  Therefore, choosing $\nu=\sqrt{\omega}$, we obtain a stationary $(\theta_+-\theta_-)$-twisted $N$-helix
  filament configuration with some localized peturbation travelling in time on
  each of its filaments.

\medskip
Last but not least, in the last part of this paper we describe configurations of nearly parallel filaments that lead to a collision in finite time. They are obtained by the same kind of dilation-rotation perturbations as in Theorem \ref{thm:main-2}.

\begin{theorem}
 \label{thm:blup}
Let $N\geq 2$ and $(X_j)_j$ be the stationnary configuration given by a regular $N-$polygon with its center and circulations $\Gamma_j=1$ for $1\leq j\leq N$, $\Gamma_0=-(N-1)/2$. Then the initial condition 
$$\Psi_{j,0}(\sigma)=X_j(0)\left(1-\frac{e^{-\frac{\sigma^2}{1-4i}}}{\sqrt{1-4i}}\right)$$
yields a solution $(\Psi_j)_j$ for Syst.  \eqref{syst:interaction}, with $\Psi_j-X_j\in C\left(\R,H^1(\R)\right)$, 
that collide at time $t=1$ at $\sigma =0$. 
\end{theorem}

\medskip

The remainder of this paper is organized as follows. In 
Section \S\ref{sec:sym} we derive Eq. \eqref{eq:BM}. We then present some 
preliminary lemmas about its energy, which lead to the proof of Theorem
 \ref{thm:main-2}. Section \S\ref{sec:losange} is devoted to the proof of Theorem \ref{thm:main}.  
Section \S\ref{sec:tw} contains the construction of travelling waves for 
Theorem \ref{thm:TW}. Finally, in Section \S\ref{sec:blup} we construct the collision dynamics in Theorem \ref{thm:blup}. In all the following the notation $C$ denotes an absolute constant which can possibly change from
a line to another.

\section{Proof of Theorem \ref{thm:main-2}}\label{sec:sym}
We first derive Eq. \eqref{eq:BM}.  Plugging the ansatz 
$\Psi_j(t,\sigma)=X_j(t)\Phi(t,\sigma)$ into Syst. \eqref{syst:interaction} with $\Gamma_j=1$ for $1\leq j\leq N$ 
we obtain 
$$iX_j\dt \Phi+i\dt X_j \Phi+X_j \partial_\sigma^2 \Phi +\frac{\Phi}{|\Phi|^2}\sum_{k\neq j} 
\frac{X_j-X_k}{|X_j-X_k|^2}=0.$$
Next we use \eqref{syst:point-vortex} to get
$$X_j(i\dt \Phi+ \partial_\sigma^2 \Phi )-i\dt X_j\frac{\Phi}{|\Phi|^2}\left(1-|\Phi|^2\right)=0.$$
Now if we consider a configuration rotating with speed $\omega$ around its steady center of inertia $X_0=0$, 
for $1\leq j\leq N$ we have  $X_j(t)=e^{it\omega+i\theta_j}$, so that $-i\partial_t X_j=\omega X_j$ and hence
 we obtain Eq. \eqref{eq:BM},
$$i\dt \Phi+\partial_\sigma^2 \Phi_j+\omega\frac{\Phi}{|\Phi|^2}\left(1-|\Phi|^2\right)=0.$$
Conversely, assume that $\Phi$ is a solution to Eq. \eqref{eq:BM} and set 
$\Psi_j=X_j\Phi$. Reversing the previous arguments, we obtain
$$i\dt \Psi_j+ \partial_\sigma^2 \Psi_j +\sum_{k\neq j} 
\frac{\Psi_j-\Psi_k}{|\Psi_j-\Psi_k|^2}=0,\quad 1\leq j\leq N,$$
while, since $\Psi_0(t,\sigma)=0$ for all $(t,\sigma)$, 
$$i\dt \Psi_0=\partial_\sigma^2 \Psi_0=0\quad \text{and}\quad  \sum_{k=1}^N \frac{\Psi_0-\Psi_k}{|\Psi_0-\Psi_k|^2}=
-\frac{\Phi}{|\Phi|^2}\sum_{k=1}^N \frac{X_k}{|X_k|^2}=0$$
and therefore $(\Psi_j)_j$
 is a solution to Syst. \eqref{syst:interaction}.

\subsection{Some preliminary lemmas}

\begin{lemma}\label{lemma:ginzburg}There exists an absolute constant $\eta_1$ and a time $t_1$ depending only on $\eta_1$ such that\\
i) If $\mathcal E(f)\leq \eta_1$ then 
$$\||f|^2-1\|_{L^\infty}\leq \frac 14.$$
ii) If $\|\partial_\sigma f\|_{L^2}\leq \eta_1$ then for all $0\leq t\leq t_1$
$$\frac{1}{\sqrt{2}}\|e^{it\partial_\sigma^2}f-f\|_{L^\infty}\leq \|e^{it\partial_\sigma^2}f-f\|_{H^1}\leq \frac 14.$$
\end{lemma}
\begin{proof}
i) The function $a(x)=x-1-\ln x$ is positive and convex, and vanishes only at  $x=1$, therefore we can adapt 
standard arguments already used in the context of Ginzburg-Landau-type functionals (see e.g. \cite{BBH}). 
More precisely,
we assume by contradiction that  $\left||f(\sigma_0)|^2-1\right|>1/4$ for some $\sigma_0\in \R$.
For example, $|f(\sigma_0)|>\sqrt{5/4}$. Next, since $\|\partial_\sigma f\|_{L^2}^2\leq 2\mathcal E(f)$ 
we have
by Cauchy-Schwarz inequality
\begin{equation*}
 |f(\sigma)|\geq |f(\sigma_0)|-\left|\int_{\sigma_0}^{\sigma} \partial_xf(x)dx\right|
\geq \sqrt{\frac{{5}}{4}}-\sqrt{2\mathcal E(f)|\sigma-\sigma_0|}.
\end{equation*}
It follows that $|f|>\sqrt{9/8}$ on $I=[\sigma_0-1/(500\mathcal{E}(f)),\sigma_0+1/(500 \mathcal{E}(f))]$. Therefore
\begin{equation*}
 \mathcal{E}(f)\geq \frac{1}{2}a\left(\frac{9}{8}\right)|I|=\frac{1}{500 \mathcal E(f)}a\left(\frac{9}{8}\right),
\end{equation*}
a contradiction if $\mathcal E(f)\leq \eta_1$ is sufficiently small.

ii) The property ii) is a known one used in 
the Gross-Pitaevskii study (see Lemma 3 in \cite{PG0}) to which we recall the short proof: the Fourier transform of $e^{it\partial_\sigma^2}f-f$ can 
be written as $\frac{e^{-it\xi^2}-1}{\xi}\,\xi\hat f(\xi)$, 
so the $L^2$ norm is bounded by $C\sqrt{t}\|\partial_\sigma f\|_{L^2}$ and the $\dot H^1$ norm 
is bounded by $C\|\partial_\sigma f\|_{L^2}$, i.e. 
$$\|e^{it\partial_\sigma^2}f-f\|_{H^1}\leq C(1+\sqrt{t})
\|\partial_\sigma f\|_{L^2}\leq  C(1+\sqrt{t}) \eta_1.$$
We choose $\eta_1$ small enough and $t_1$ small with respect to $ \eta_1$ such that for  $0\leq t\leq t_1$,
$$\|e^{it\partial_\sigma^2}f-f\|_{H^1}\leq \frac 14,$$
and the conclusion of the Lemma follows.
\end{proof}

Since $(x-1)^2/2\leq x-1-\ln x\leq 10(x-1)^2$ on $[3/4,5/4]$ we immediately obtain a second lemma.
\begin{lemma}\label{lemma:comp}If $\||f|^2-1\|_{L^\infty}\leq 1/4$ then we can compare the energies: 
$$\mathcal E_{GP}(f)
\equiv\frac{1}{2}\|\partial_\sigma f\|_{L^2}^2+\frac{\omega}{4}\||f|^2-1\|_{L^2}^2\leq 
\mathcal E(f)\leq 5\,\mathcal E_{GP}(f).$$
\end{lemma}
So, if we consider an initial perturbation such that  $\Phi_0-1$ is sufficiently small in $H^1$, we infer from Sobolev 
embedding that
 $\mathcal E_{GP}(\Phi_0)<\infty$ and that $\| |\Phi_0|^2-1\|_{L^\infty}<1/4$. Hence Lemma \ref{lemma:comp} 
ensures that $\Phi_0$ belongs to the energy space associated to Eq. \eqref{eq:BM}.

\medskip

We will also need the following transposition of a standard property of the Gross-Pitaevskii energy (see \cite{Ga, PG0,PG}).
\begin{lemma}
 \label{lemma:recall}
Let $f$ such that $\mathcal{E}(f)\leq \eta_1$, with $\eta_1$ defined in Lemma \ref{lemma:ginzburg}. 
Let $h\in H^1(\R)$ with $\|h\|_{H^1}\leq 1/2$. Then  the energy $\mathcal{E}(f+h)$ is finite. More precisely we have,
for absolute numerical constants $C,C'$,
\begin{equation*}\begin{split}
\mathcal{E}(f+h)&\leq C\mathcal{E}_{GP}(f+h)
\leq C'\left(1+\mathcal{E}(f)\right)\left(1+\|h\|_{H^1}^2\right).\end{split}\end{equation*}  Moreover, 
\begin{equation*}
 \||f+h|-1\|_{L^\infty}\leq \frac{2+\sqrt{2}}{4}<1.
\end{equation*}
\end{lemma}
\begin{proof}
 We first infer from Lemma \ref{lemma:ginzburg} i) that $\||f|-1\|_{L^\infty}\leq 1/4$, and from Lemma \ref{lemma:comp} that
$\mathcal{E}_{GP}(f)<\infty$. Next, applying Gagliardo-Nirenberg 
inequality we get 
$\|h\|_{L^\infty}\leq \sqrt{2}\|h\|_{H^1}\leq \sqrt{2}/2$, so that $\||f+h|-1\|_{L^\infty}\leq (2+\sqrt{2})/4<1$. By 
Lemma \ref{lemma:comp} it follows that
$\mathcal{E}(f+h)\leq C\mathcal{E}_{GP}(f+h)$. 
Using that  $\mathcal{E}_{GP}(f)<\infty$ and $h\in H^1$ as well as Sobolev inequalities 
 we conclude that $\mathcal{E}_{GP}(f+h)$ is finite, with the corresponding estimate 
(see also, e.g., Lemma 2 in \cite{PG0}). 
\end{proof}

\subsection{Proof of Theorem \ref{thm:main-2}}

First we will establish local well-posedness for Eq. \eqref{eq:BM} by performing a fixed point argument for the operator
$$A(w)(t)=i\int_0^t e^{i(t-\tau)\partial_\tau^2}
\frac{e^{i\tau\partial_\sigma^2}\Phi_0+w(\tau)}{|e^{i\tau\partial_\sigma^2}\Phi_0+
w(\tau)|^2}\left(1-|e^{i\tau\partial_\sigma^2}\Phi_0+w(\tau)|^2\right)\,d\tau$$
on the ball
$$B_T
=\left\{ w\in C\left([0,T],H^1\right),\quad \sup_{0\leq t\leq T}\|w(t)\|_{H^1}\leq\frac 14\right\},$$
with $T$ small to be chosen later. Then $\Phi(t)=e^{it\partial_\sigma^2}\Phi_0+w(t)$ will be a 
solution for \eqref{eq:BM} on $[0,T]$  with initial data $\Phi_0$. Observe that the proof of Lemma \ref{lemma:ginzburg} ii)
yields that $t\mapsto (e^{it\partial_\sigma^2 }\Phi_0-\Phi_0)\in C([0,T],H^1(\R))$.
So the map $\Phi$ will belong to
 the energy space if $\Phi_0$ belongs to the energy space (by Lemma \ref{lemma:recall}
applied to $f=\Phi_0$ and $h=e^{it\partial_\sigma^2}\Phi_0-\Phi_0+w(t)$ 
for $T\leq t_1$ with $t_1$ from Lemma \ref{lemma:ginzburg}), and it will belong
to $1+H^1(\R)$ if $\Phi_0$ is in $1+H^1(\R)$. 

\medskip

The hypothesis of Theorem \ref{thm:main-2} is that we start with $\Phi_0$ verifying
 $$\mathcal E=\mathcal E(\Phi_0)= \frac{1}{2}\|\partial_\sigma \Phi_0\|_{L^2}^2+\frac{\omega}{2}\int
\left(-\ln |\Phi_0|+|\Phi_0|^2-1\right)\leq \eta_1.$$
We first impose $T\leq t_1$, with $t_1$ defined in Lemma \ref{lemma:ginzburg}. Let $w\in B_T$, and set for $0\leq t\leq T$
\begin{equation*}
 \tilde{\Phi}(t)=e^{it\partial_\sigma^2}\Phi_0+w(t)
=\Phi_0+\left(e^{it\partial_\sigma^2}\Phi_0-\Phi_0+w(t)\right).
\end{equation*}
By Lemma \ref{lemma:ginzburg} ii) and by choice of $B_T$, we have 
 $\|\tilde{\Phi}(t)-\Phi_0\|_{H^1}\leq 1/2$ on $[0,T]$.
Therefore, applying Lemma \ref{lemma:recall} to $f=\Phi_0$ and
$h=\tilde{\Phi}(t)-\Phi_0$ we obtain that $
\||\tilde{\Phi}(t)|-1\|_{L^\infty}\leq (2+\sqrt{2})/4$ on $[0,T]$.
In particular, since $C^{-1}\leq|\tilde{\Phi}|\leq C$ for $C>0$ we can 
estimate the action of the operator as follows
\begin{equation*}\begin{split}\|A&(w)(t)\|_{H^1}\leq t\sup_{0\leq \tau\leq t}\left\|\frac{\tilde{\Phi}(\tau)}{|\tilde{\Phi}(\tau)|^2}\left(1-|\tilde{\Phi}(\tau)|^2\right)
\right\|_{H^1}\\
&\leq C\,t\sup_{0\leq \tau\leq t}\left(\left\|1-|\tilde{\Phi}(\tau)|^2\right\|_{L^2}
+\left\|\partial_\sigma \tilde{\Phi}(\tau)\right\|_{L^2}\right)\\
&\leq C\,t\sup_{0\leq \tau\leq t} \sqrt{\mathcal{E}_{GP}}(\tilde{\Phi}(\tau)).\end{split}\end{equation*}
We use again Lemma  \ref{lemma:recall} and the bound 
$\|\tilde{\Phi}(\tau)-\Phi_0\|_{H^1}\leq 1/2$ to obtain
\begin{equation*}\begin{split}\sup_{0\leq t\leq T}\|A(w)(t)\|_{H^1}&\leq
 C\,T(1+\mathcal E).\end{split}\end{equation*}
Arguing similarly, we readily check that for $w_1,w_2\in B_T$
\begin{equation*}\begin{split}
\sup_{0\leq t\leq T}
\|A(w_1)(t)-A(w_2)(t)\|_{H^1}&\leq C\,T(1+\mathcal E)\sup_{0\leq t\leq T}\|w_1(t)-w_2(t)\|_{H^1}.\end{split}\end{equation*}
Hence imposing a second smallness condition on $T$ with respect to
 $\mathcal E$ we obtain a 
fixed point $w$ for $A$ in $B_T$. Therefore local well-posedness holds for Eq. \eqref{eq:BM} on $[0,T]$ with $T$ depending
only on $\mathcal{E}$.

\medskip

Next, since the energy of 
Eq. \eqref{eq:BM} is conserved
$$\mathcal E(\Phi(T))=\mathcal E(\Phi(0))=\mathcal E,$$
we re-iterate the local in time argument to get the global existence. Finally, Lemma \ref{lemma:ginzburg} insures us that
$$\sup_{t\in \R}\||\Phi(t)|^2-1\|_{L^\infty}\leq \frac 14,$$
so the solution satisfies indeed
$$\frac 14\leq |\Phi(t,\sigma)|\leq \frac 54,\quad t,\sigma\in \R.$$

\section{Proof of Theorem \ref{thm:main}}
\label{sec:losange}

\subsection{Some useful quantities}
From now on we will write
 $\P_{jk}=\P_j-\P_k$, 
$X_{jk}=X_j-X_k$ and $u_{jk}=u_j-u_k$.

We first introduce some useful quantities.  In the general 
case where $N\geq 1$ and $\Gamma_j\in \R$, 
the dynamics of Syst. \eqref{syst:interaction} preserves the following quantities. 

The energy 
\begin{equation*}
\frac{1}{2} \sum_{j}
  \Gamma_j^2 \int \left|\partial_\sigma \P_j(t,\sigma)\right|^2\,d\sigma
-\frac{1}{2}\sum_{j\neq k} \Gamma_j \Gamma_k \int \ln \left|\P_{jk}(t,\sigma)\right|^2\,d\sigma,
\end{equation*}
the angular momentum
\begin{equation*}
 \sum_j \Gamma_j \int \left|\P_j(t,\sigma)\right|^2\,d\sigma,
\end{equation*}
and
\begin{equation*}
\sum_{j\neq k}\Gamma_j \Gamma_k \int \left|\P_{jk}(t,\sigma)\right|^2\,d\sigma.
\end{equation*}

However the previous quantities are not well-defined in the framework of Theorem \ref{thm:main}, not even formally, since 
 $\Psi_j(t,\sigma)$ and 
$\Psi_{jk}(t,\sigma)$ do not tend to zero at infinity. As in \cite{KPV}, we modify them in order to get well-defined
quantities, introducing
\begin{equation*}\begin{split}
 \mathcal{H}&=\frac{1}{2}\sum_{j}
  \Gamma_j^2 \int \left|\partial_\sigma \P_j(t,\sigma)\right|^2\,d\sigma
-\frac{1}{2}\sum_{j\neq k} \Gamma_j \Gamma_k \int \ln\left(\frac{|\P_{jk}(t,\sigma)|^2}{|X_{jk}(t)|^2}\right)\,d\sigma\\
 \mathcal{A}&=\sum_j \Gamma_j \int \left(|\P_j(t,\sigma)|^2-|X_j(t)|^2\right)\,d\sigma\\
 \mathcal{T}&=\sum_{j\neq k}\Gamma_j \Gamma_k \int \left(|\P_{jk}(t,\sigma)|^2-|X_{jk}(t)|^2\right)\,d\sigma.\end{split}
\end{equation*}
Note that, in view of the properties of the point vortex system \eqref{syst:point-vortex} mentioned
in the introduction, the renormalized quantities $\mathcal{H}$, $\mathcal{A}$ and $\mathcal{T}$
are still formally preserved in time.

Finally, we also introduce the time-dependent quantity
\begin{equation*}
 \mathcal{I}(t)=\frac{1}{2}\sum_{j\neq k}\Gamma_j \Gamma _k \int \left( \frac{|\P_{jk}(t)|^2}{|X_{jk}(t)|^2}-1\right)\,d\sigma,
\end{equation*}
and we consider the energy 
\begin{equation}\label{def:energy-losange-bis}
\mathcal E(t)=\mathcal H+\mathcal I(t),\end{equation}
which have been already introduced in \eqref{def:energy-losange} in the introduction.

As noticed in \cite{KPV}, a useful consequence of the convexity estimate 
$(x-1)^2/2\leq x-1-\ln x\leq 10(x-1)^2$ on $[3/4,5/4]$ is the inequality
\begin{equation}
 \label{ineq:plus-loin}
\frac{1}{2}\sum_j  \Gamma_j^2 \int \left|\partial_\sigma \P_j(t,\sigma)\right|^2\,d\sigma
+\frac{1}{4}\sum_{j\neq k}\Gamma_j\Gamma_k \int \left(\frac{|\P_{jk}(t,\sigma)|^2}{|X_{jk}(t)|^2}-1\right)^2\,d\sigma
\leq \mathcal{E}(t),
\end{equation}
which holds as long as the filaments satisfy $3/4\leq |\Psi_{jk}(t)|^2/|X_{jk}(t)|^2\leq 5/4$. 

\subsection{The approach}
In this subsection we briefly sketch how to combine 
elements from \cite{KPV} and from \S2 to prove local existence and uniqueness of a solution to 
Syst. \eqref{syst:interaction} in the 
general case of $N$ filaments, with $N\geq 2$, and the way to extend this solution as long as the energy $\mathcal E(t)$
remains sufficiently small. Here we take positive circulations
$$\Gamma_j>0,\quad 1\leq j\leq N.$$
Therefore there exists a unique global solution $(X_j)_j$ to Syst. 
\eqref{syst:point-vortex}. We denote by $d>0$  the minimal 
distance between the point vortices for all time. Here we shall make the extra-assumption that 
$$u_{j,0}=\Psi_{j,0}-X_{j,0}\in H^1(\R).$$

\medskip

We look for a solution $u=(u_j)_j\in C([0,T],H^1(\R))^N$ to the system
\begin{equation}
 \label{syst:pert-u}
\begin{cases}
\displaystyle i\partial_t u_j + \Gamma_j\partial_\sigma^2 u_j+\sum_{k\neq j} \Gamma_k 
\left(\frac{ X_{jk}+u_{jk}}{|X_{jk}+u_{jk}|^2}-\frac{ X_{jk}}{|X_{jk}|^2}\right)=0\\
\displaystyle u_j(0)=u_{j,0},\quad 1\leq j\leq N. 
\end{cases}
\end{equation}
By similar arguments as in Section \S\ref{sec:sym}, our purpose is to find a fixed point in the Banach space
\begin{equation*}
B_T=\left\{ w=(w_1,\ldots,w_N)\in C\left([0,T],H^1\right)^N,\quad \sup_{0\leq t\leq T} \|w(t)\|_{H^1}
\leq \frac{d}{4}\right\} 
\end{equation*}
for the operator $A(w)=(A_j(w))_j$ defined by
\begin{equation*}
A_j(w)(t)=i\int_0^t 
\sum_{k\neq j} \Gamma_k 
\left(\frac{ X_{jk}(\tau)+e^{i\tau  \Gamma_j\partial_\sigma^2}u_{j,0}+w_{j}(\tau)-e^{i\tau  \Gamma_k\partial_\sigma^2}u_{k,0}-w_{k}(\tau)}
{|X_{jk}(\tau)+e^{i\tau  \Gamma_j\partial_\sigma^2}u_{j,0}+w_{j}(\tau)-e^{i\tau  \Gamma_k\partial_\sigma^2}u_{k,0}-w_{k}(\tau)|^2}
-\frac{ X_{jk}(\tau)}{|X_{jk}(\tau)|^2}\right)\,d\tau,
\end{equation*}
and for $T$
sufficiently small with respect to $\eta_2$,  $\sum_j \|u_{j,0}\|_{H^1},$ $(\Gamma_j)_j$  and $d$.

Then as in Section \S\ref{sec:sym} the solution will be given by
\begin{equation*}
 u_{j}(t)=e^{it \Gamma_j\partial_\sigma^2}u_{j,0}+w_{j}(t).
\end{equation*}

By transposing the arguments of Section \S\ref{sec:sym} we obtain the following local well-posedness result.
\begin{lemma}\label{lemma:lwp}
 Let $(u_{j,0})_j\in H^1(\R)^N$ be such that $\mathcal{E}_0< 10\eta_2$, with $\mathcal{E}_0$ defined in Theorem 
\ref{thm:main} and $\eta_2=\eta_2(d)$ a small constant depending only on $d$.
 There exists $T>0$, depending only on $\eta_2$, $\sum_j \|u_{j,0}\|_{H^1}$, $(\Gamma_j)_j$ and $d$,  and there exists 
a unique solution $(u_j)_j\in C([0,T],H^1(\R))^N$ to Syst. \eqref{syst:pert-u} satisfying
\begin{equation*}
\sup_{0\leq t\leq T}  \|u_j(t)\|_{H^1}\leq \|u_{j,0}\|_{H^1}+\frac{d}{4},\quad 1\leq j\leq N.
\end{equation*}
 
Moreover we can choose $T$ such that 
\begin{equation*}
 T\big(1+\eta_2+\sum_j\|u_{j,0}\|_{H^1}\big)\geq C(d,(\Gamma_j)_j)
\end{equation*}
for some  constant $C(d,(\Gamma_j)_j)$ depending only on $d$ and $(\Gamma_j)_j$.

\end{lemma}

\begin{remark}\label{rem:extension}
As a byproduct of Lemma \ref{lemma:lwp} we realize that the solution $(u_j)_j$ to \eqref{syst:pert-u} 
exists as long as the energy $\mathcal{E}(t)$ remains bounded by $10\eta_2$. 
Indeed note that
the norm $\sum_j\|u_j(t)\|_{H^1}$ can grow exponentially, but it cannot blow up as long as the energy is
sufficiently small.
\end{remark}

\begin{proof} Let 
$0<\Gamma\leq 1$ such that $0<\Gamma\leq \min_{j} \Gamma_j$.
Since all the $(\Gamma_j)'s$ are positive, we have
\begin{equation*}
 \max_{j\neq k}
\mathcal{E}\left(\frac{\Psi_{jk,0}}{X_{jk,0}}\right)\leq \frac{1}{\Gamma^2} \mathcal{E}_0,
\end{equation*} 
where we recall that  $\mathcal{E}$ is  defined by \eqref{def:energy-BM} (taking $\omega=1$).

In particular, if $\eta_2$ is such that
$10 \eta_2/\Gamma^2 \leq \eta_1$,
with $\eta_1$ defined in Lemma \ref{lemma:ginzburg}, then   
$3/4 \leq |\Psi_{jk,0}|/|X_{jk,0}|\leq 5/4$ for all $j\neq k$. 
Then we have for $w\in B_T$ 
\begin{equation*}
\begin{split}
&|X_{jk}(\tau)+e^{i\tau  \Gamma_j\partial_\sigma^2}u_{j,0}+w_{j}(\tau)-e^{i\tau  \Gamma_k\partial_\sigma^2}u_{k,0}-w_{k}(\tau)|\\
=&|\Psi_{jk,0}+(X_{jk}(\tau)-X_{jk,0})+(e^{i\tau \Gamma_j\partial_\sigma^2}u_{j,0}-u_{j,0})-(e^{i\tau \Gamma_k\partial_\sigma^2}u_{k,0}-u_{k,0})+w_{jk}(\tau)|\\
\geq&  |\Psi_{jk,0}|-|X_{jk}(\tau)-X_{jk,0}|-\sqrt{2}
\|(e^{i\tau \Gamma_j\partial_\sigma^2}u_{j,0}-u_{j,0})-(e^{i\tau \Gamma_k\partial_\sigma^2}u_{k,0}-u_{k,0})+w_{jk}(\tau)\|_{H^1}\\
\geq& \frac{3d}{4}-\frac{2(\sum_{j} \Gamma_j)}{d}T-C(1+T)\eta_2-\frac{\sqrt{2}d}{4}\geq \frac{d}{4}
\end{split}
\end{equation*}
provided that $\eta_2$ is small with respect to $d$, 
and that $T$ is small in terms of $\eta_2,d, (\Gamma_j)_j$.
In the last inequality we have used the proof of Lemma \ref{lemma:ginzburg} ii) together 
with the mean-value theorem for $X_{jk}$.
Now, since $X_{jk}(\tau)+e^{i\tau  \Gamma_j\partial_\sigma^2}u_{j,0}+w_{j}(\tau)-e^{i\tau  \Gamma_k\partial_\sigma^2}u_{k,0}-w_{k}(\tau)$ is bounded from below, direct estimates 
show that $A$ is a contraction on $B_T$ as long as
\begin{equation*}
 T(1+\eta_2+\sum_j\|u_{j,0}\|_{H^1})\leq C(d,(\Gamma_j)_j)
\end{equation*}
and the conclusion of  Lemma \ref{lemma:lwp} follows.

\end{proof}

\subsection{The proof of Theorem \ref{thm:main}} 
We present now the proof of Theorem \ref{thm:main}. By Remark \ref{rem:extension}, there exists a unique 
solution as long as $\mathcal E(t)$ remains sufficiently small.
 In the cases considered in \cite{KPV} where the  $|X_{jk}(t)|$ are all  the same and constant  equal to $d$, 
$\mathcal I(t)=\mathcal T/(2d^2)$ so $\mathcal E(t)$ is conserved. Also
 under the hypothesis of Theorem \ref{thm:main-2}, we have 
$$\mathcal{I}(t)=\frac{1}{2}\sum_{j\neq k}\Gamma_j\Gamma_k 
\int \left( \frac{|\P_{jk}(t)|^2}{|X_{jk}|^2}-1\right)\,d\sigma=\frac{1}{2}\sum_{j\neq k}\Gamma_j\Gamma_k 
\int \left(|\Phi(t,\sigma)|^2-1\right)\,d\sigma=\omega \mathcal{A},$$
so, although $|X_{jk}|$ are not all equal, $\mathcal{I}(t)$ and $\mathcal E(t)$ are still formally preserved.
In fact, under the assumptions of Theorem \ref{thm:main-2} we have
 $\mathcal{E}(t)=N \mathcal{E}(\Phi(t))$ so we retrieve the fact that it is constant.  
Under the general hypothesis of Theorem \ref{thm:main} $\mathcal{E}(t)$ is no  longer constant, 
but it will still be a useful quantity for which we can achieve some control.

\medskip
We recall that $\mathcal{E}_0\leq \eta_2$. From now on we consider $T>0$ 
and the unique solution to Syst. \eqref{syst:pert-u} on $[0,T]$, with $\mathcal{E}(t)< 10\tilde{\mathcal{E}_0}\leq 10\eta_2$, 
given by Lemma \ref{lemma:lwp}. We take $T$ maximal in the sense that  $\mathcal{E}(T)= 10\tilde{\mathcal{E}_0}$ (but $T$ is not necessarily the largest time of existence).
We thus have $3/4< |\Psi_{jk}(t,\sigma)|< 5/2$  on $[0,T]\times \R$ for all $j\neq k$.

\begin{proposition}\label{prop:I}
 We have for $t\in [0,T]$
\begin{equation*}
 \mathcal{E}(t)=\mathcal H+\frac{1}{2}\mathcal{T}
-\mathcal{A}+
\frac{\|(u_1+u_3)(t)\|^2+ \|(u_2+u_4)(t)\|^2}{2}.
\end{equation*}

\end{proposition}

\begin{proof}Since $(X_1,X_2,X_3,X_4)$ is a square of radius $1$ we have 
\begin{equation*}
 |X_{jk}(t)|^2=2\quad \text{if } |j-k|=1,\quad |X_{jk}(t)|^2=4\quad \text{if }|j-k|=2.
\end{equation*}
It follows that
\begin{equation*}
 \begin{split}
&  \sum_{j\neq k} \left(\frac{|\P_{jk}|^2}{|X_{jk}|^2}-1\right)
=\sum_{j\neq k}\frac{|\P_{jk}|^2-|X_{jk}|^2}{|X_{jk}|^2}\\
&=\frac{1}{2}\sum_{j\neq k}\left(|\P_{jk}|^2-|X_{jk}|^2\right)+2\left(\frac{1}{4}-\frac{1}{2}\right)\left( |\P_{13}|^2-|X_{13}|^2+|\P_{24}|^2-|X_{24}|^2\right).
 \end{split}
\end{equation*}
On the other hand, we compute
\begin{equation*}
 \begin{split}
  & |\P_{13}|^2+|\P_{24}|^2-|X_{13}|^2-|X_{24}|^2\\
&=2\sum_{j=1}^4 |\P_j|^2-|\P_1+\P_3|^2-|\P_2+\P_4|^2-8\\
&=2\sum_{j=1}^4 \left( |\P_j|^2-|X_j|^2\right)-\left( |\P_1+\P_3|^2+|\P_2+\P_4|^2\right), \end{split}
\end{equation*}
so integrating with respect to $\sigma$ and using that $\Psi_1+\Psi_3=u_1+u_3$ and $\Psi_2+\Psi_4=u_2+u_4$ we are led to the
conclusion.
\end{proof}

\begin{corollary}
In the case of the parallelogram $\|(u_1+u_3)(0)\|_{L^2}^2=\|(u_2+u_4)(0)\|_{L^2}^2=0$, 
so it follows that $\|(u_1+u_3)(t)\|_{L^2}^2=\|(u_2+u_4)(t)\|_{L^2}^2=0$ for all times, using 
the fact that if $(\Psi_1,\Psi_2,\Psi_3,\Psi_4)$ is a solution of \eqref{syst:interaction} 
then $(-\Psi_3,-\Psi_4,-\Psi_1,-\Psi_2)$ is also a solution. Then $\mathcal I$ is conserved in time and 
global existence follows. 
\end{corollary}

\begin{remark}
One can do similar computations in others particular cases, for instance for 
ends and the middle of the segment,
$$\mathcal E(t)=-\mathcal{H}+\mathcal{I}-\frac 32 \mathcal{A} + \frac 34 \left(\|u_1(t)\|_{L^2}^2
+\|(u_2+u_3)(t)\|_{L^2}^2\right),$$
or for hexagone,
$$\mathcal E(t)=-\mathcal{H}+\mathcal{I}-\frac 72 \mathcal{A} + 
\frac 23 \sum_{j=1}^2\|(u_j+u_{j+2}+u_{j+4})(t)\|_{L^2}^2+\frac 34 \sum_{j=1}^3 \|(u_j+u_{j+3})(t)\|_{L^2}^2.$$
But these quantities have no reason to be conserved, unless the perturbations have the same shape as the shape of $(X_j)$, which enters the framework of the first part of this article. Moreover, when trying to control the growth of $\|u_1(t)\|_{L^2}$ for instance in the first example, the time of control is not satisfactory due to the presence of linear terms in the equation of $u_1$, that cannot be resorbed.
\end{remark}

In order to control the evolution of the energy we have to
control the quantity $\|(u_1+u_3)(t)\|_{L^2}^2+\|(u_2+u_4)(t)\|_{L^2}^2$. We are led to introduce the new unknowns
$$v=u_1+u_3,\quad w=u_2+u_4.$$

\begin{proposition}\label{prop:est-U}
 We have for $t\in [0,T]$, with $v=u_1+u_3$ and $w=u_2+u_4$,
\begin{equation*}\begin{split}
\|v(t)\|_{L^2}+\|w(t)\|_{L^2}
&\leq \|v(0)\|_{L^2}+\|w(0)\|_{L^2}\\+ Ct\sup_{s\in [0,T]}&
 \max_{j\neq k}\|u_{jk}(s)\|_{L^2}^{1/2}\mathcal E(s)^{1/4}
\left(\|v(s)\|_{L^2}+\|w(s)\|_{L^2}+
\mathcal E(s)^{1/2}\right).\end{split}
\end{equation*}

\end{proposition}

\begin{proof}In view of Syst. \eqref{syst:interaction} and Syst. \eqref{syst:point-vortex}, we have
\begin{equation*}
 \begin{split}
  i\partial_t v+\partial_\sigma^2 v&
=-\sum_{k\neq 1,3}\left\{
\left( \frac{\P_{1k}}{|\P_{1k}|^2}-\frac{X_{1k}}{|X_{1k}|^2}\right)
+\left(\frac{\P_{3k}}{|\P_{3k}|^2}-\frac{X_{3k}}{|X_{3k}|^2}\right)\right\}\\
&=-\sum_{k\neq 1,3}
\left\{ X_{1k}\left( \frac{1}{|\P_{1k}|^2}-\frac{1}{|X_{1k}|^2}\right)
+X_{3k}\left( \frac{1}{|\P_{3k}|^2}-\frac{1}{|X_{3k}|^2}\right)\right\}\\
&-\sum_{k\neq 1,3}
\left\{ u_{1k}\left( \frac{1}{|\P_{1k}|^2}-\frac{1}{|X_{1k}|^2}\right)
+u_{3k}\left( \frac{1}{|\P_{3k}|^2}-\frac{1}{|X_{3k}|^2}\right)\right\}\\
&-\sum_{k\neq 1,3}\left\{ \frac{u_{1k}}{|X_{1k}|^2}+\frac{u_{3k}}{|X_{3k}|^2}\right\}.
 \end{split}
\end{equation*}
We infer that
\begin{equation*}
 \begin{split}
  i\partial_t v+\partial_\sigma^2 v&=\mathcal{L}_v(u)+\mathcal{R}_v(u),\end{split}\end{equation*}
where $\mathcal{L}_v$ denotes the linear part, 
\begin{equation*}
\begin{split}
\mathcal{L}_v&(u)=2\sum_{k\neq 1,3}
\left\{X_{1k} \frac{\re (\overline{u_{1k}}X_{1k})}{|X_{1k}|^4}
+X_{3k} \frac{\re (\overline{u_{3k}}X_{3k})}{|X_{3k}|^4}\right\}
-\sum_{k\neq 1,3}\left\{ \frac{u_{1k}}{|X_{1k}|^2}+\frac{u_{3k}}{|X_{3k}|^2}\right\}
\end{split}
\end{equation*}
and where the remainder $\mathcal{R}_v$ is quadratic in $u$,
\begin{equation*}
\begin{split}
&\mathcal{R}_v(u)\\
&=\sum_{k\neq 1,3}\left\{ \frac{X_{1k}}{|X_{1k}|^4}|u_{1k}|^2+\frac{X_{3k}}{|X_{3k}|^4}|u_{3k}|^2\right\}\\
&-\sum_{k\neq 1,3} \left\{ X_{1k}\left( \frac{|X_{1k}|^2-|\Psi_{1k}|^2}{|X_{1k}|^2}\right)
\left(\frac{1}{|\P_{1k}|^2}-\frac{1}{|X_{1k}|^2}\right)
+X_{3k}\left( \frac{|X_{3k}|^2-|\Psi_{3k}|^2}{|X_{3k}|^2}\right)
\left(\frac{1}{|\P_{3k}|^2}-\frac{1}{|X_{3k}|^2}\right)\right\}\\
&-\sum_{k\neq 1,3}
\left\{ u_{1k}\left( \frac{1}{|\P_{1k}|^2}-\frac{1}{|X_{1k}|^2}\right)
+u_{3k}\left( \frac{1}{|\P_{3k}|^2}-\frac{1}{|X_{3k}|^2}\right)\right\}\\
&=\mathcal{R}_v^1(u)+\mathcal{R}_v^2(u)+\mathcal{R}_v^3(u).
 \end{split}
\end{equation*}

\medskip

We claim that $\mathcal{L}_v(u)=0$. Indeed, using that $|X_{1k}|^2=|X_{3k}|^2=2$ for $k\neq 1,3$,
\begin{equation*}\begin{split}
 \mathcal{L}_v(u)&=
\frac{1}{2}\sum_{k\neq 1,3}
\left(X_{1k} \re (\overline{u_{1k}}X_{1k})
+X_{3k} \re (\overline{u_{3k}}X_{3k})\right)
-\frac{1}{2}\sum_{k\neq 1,3}\left(v-2u_k\right)\\
&=\frac{1}{2}\sum_{k\neq 1,3}
\left(X_{1k} \re (\overline{u_{1k}}X_{1k})
+X_{3k} \re (\overline{u_{3k}}X_{3k})\right)-v+w.
\end{split}
\end{equation*}

Now we compute, using that $X_{12}=-X_{34}$ and $X_{23}=X_{14}$,
\begin{equation*}
 \begin{split}
 &\sum_{k\neq 1,3}
\left(X_{1k} \re (\overline{u_{1k}}X_{1k})
+X_{3k} \re (\overline{u_{3k}}X_{3k})\right)\\
 &=X_{12}\re (\overline{u_{12}}X_{12})
+X_{32}\re (\overline{u_{32}}X_{32})+X_{14}\re (\overline{u_{14}}X_{14})
+X_{34}\re (\overline{u_{34}}X_{34})\\
&=X_{12}\re (\overline{u_{12}}X_{12})+X_{12}\re (\overline{u_{34}}X_{12})
+X_{32}\re (\overline{u_{32}}X_{32})++X_{32}\re (\overline{u_{14}}X_{32})\\
&=X_{12}\re ((\overline{u_{12}}+\overline{u_{34}})X_{12})+X_{32}\re ((\overline{u_{32}}+\overline{u_{14}})X_{32}).
 \end{split}
\end{equation*}
We observe that
\begin{equation*}
 u_{12}+u_{34}=u_{32}+u_{14}=u_1+u_3-(u_2+u_4)= v-w.
\end{equation*}
Therefore, inserting that $iX_{12}=X_{23}$ and that $|X_{12}|^2=2$ in the previous formula we find
\begin{equation*}
 \begin{split}
& \sum_{k\neq 1,3}
\big(X_{1k} \re (\overline{u_{1k}}X_{1k})
+X_{3k} \re (\overline{u_{3k}}X_{3k})\big)\\
&=X_{12}\re ((\overline{v}-\overline{w})X_{12})-iX_{12}\im ((\overline{v}-\overline{w})X_{12})\\&=2(v-w),
 \end{split}
\end{equation*}
and finally
$
 \mathcal{L}_v(u)=0.
$

\medskip

We next estimate the remainder terms. Since $3/4<|\Psi_{jk}|<5/2$ we have
 $\left||X_{jk}|^2-|\Psi_{jk}|^2\right|\leq C|u_{jk}|$ on $[0,T]$
and therefore
\begin{equation}\label{ineq:reste1}
 \left|\mathcal{R}^2_v(u)+\mathcal{R}^3_v(u)\right|\leq C\max_{j\neq k}
|u_{jk}|\left| \frac{|\Psi_{jk}|^2}{|X_{jk}|^2}-1\right|.
\end{equation}

Expanding the first term $\mathcal{R}_v^1(u)$ and using the symmetries of $(X_1,X_2,X_3,X_4)$, we then have
\begin{equation*}
 \begin{split}
\mathcal{R}^1_v(u)&= 
 \frac{1}{4}
\sum_{k\neq 1,3}\left\{ X_{1k}|u_{1k}|^2+X_{3k}|u_{3k}|^2\right\}\\
&= \frac{1}{4}\left\{ X_{12}\left(|u_{12}|^2-|u_{34}|^2\right)+X_{14}\left(|u_{14}|^2-|u_{32}|^2\right)\right\}\\
&=\frac{1}{2}\left\{ X_{12}\re\left(\overline{u_{12}-u_{34}}\, (v-w)\right)+
X_{14}\re\left(\overline{u_{14}-u_{32}}\, (v-w)\right)\right\},
 \end{split}
\end{equation*}
so that
\begin{equation}
 \label{ineq:reste2}
\left|\mathcal{R}^1_v(u)\right|\leq C\max_{j,k}|u_{jk}||v-w|.
\end{equation}

We perform similar computations for $w$ and from \eqref{ineq:reste1}-\eqref{ineq:reste2} we infer  the estimate
\begin{equation*}\begin{split}
 \|v(t)\|_{L^2}+\|w(t)\|_{L^2}&\leq  \|v(0)\|_{L^2}+\|w(0)\|_{L^2}
+\int_0^t \left(\|\mathcal{R}_v(u)(s)\|_{L^2}+\|\mathcal{R}_w(u)(s)\|_{L^2}\right)\,ds\\
&\leq  \|v(0)\|_{L^2}+\|w(0)\|_{L^2}\\&
+t\sup_{s\in[0,t]}\max_{j\neq k}\|u_{jk}(s)\|_{L^\infty}\left(
\left\|\frac{|\P_{jk}(s)|^2}{|X_{jk}(s)|^2}-1\right\|_{L^2}+\|v(s)\|_{L^2}+\|w(s)\|_{L^2}\right).
\end{split}
\end{equation*}
Finally we apply Gagliardo-Nirenberg inequality and \eqref{ineq:plus-loin} to obtain the conclusion.
\end{proof}

\begin{proposition}\label{prop:est-uj}
 We have for $t\in [0,T]$
\begin{equation*}
\sum_{j\neq k} \|u_{jk}(t)\|_{L^2}\leq C \sum_{j\neq k} \|u_{jk}(0)\|_{L^2}
+C\,t\, \sup_{s\in[0,t]}\mathcal E(s)^{1/2}.
\end{equation*}

\end{proposition}

\begin{proof}
By \eqref{syst:pert-u},
\begin{equation*}
 \begin{split}
 & i\partial_t u_{jk}+\partial_\sigma^2 u_{jk}\\
&=-\sum_{l\neq j}\frac{u_{jl}}{|\Psi_{jl}|^2}+\sum_{l\neq k}\frac{u_{kl}}{|\Psi_{kl}|^2}
-\sum_{l\neq j} X_{jl}\left( \frac{1}{|\Psi_{jl}|^2}
-\frac{1}{|X_{jl}|^2}\right)+\sum_{l\neq k} X_{kl}\left( \frac{1}{|\Psi_{kl}|^2}
-\frac{1}{|X_{kl}|^2}\right).
 \end{split}
\end{equation*}
We multiply the equation by $\overline{u_{jk}}$, take the imaginary part and perform the sum over $j$ and $k$, cancelling
 the first two terms in the right-hand side. Indeed,
\begin{equation*}
 \begin{split}
\sum_{j, k} \sum_{l\neq j} \frac{\im\left( u_{jk}\overline{u_{jl}}\right)}{|\Psi_{jl}|^2}
&=\sum_{j, k} \sum_{l\neq j} 
\frac{\im\left( (u_{jl}+u_{lk})\overline{u_{jl}}\right)}{|\Psi_{jl}|^2}\\
&=\sum_{j, k} \sum_{l\neq j} \frac{\im\left( u_{lk}\overline{u_{jl}}\right)}{|\Psi_{jl}|^2}\\
&=-\sum_{j, k} \sum_{l\neq j} \frac{\im\left( u_{jk}\overline{u_{jl}}\right)}{|\Psi_{jl}|^2},
\end{split}
\end{equation*}
by exchanging $j$ and $l$ in the last equality. Therefore the latter sum vanishes. By the same arguments we
also have
\begin{equation*}
 \sum_{j, k} \sum_{l\neq k} \frac{\im\left( u_{jk}\overline{u_{kl}}\right)}{|\Psi_{kl}|^2}=0.
\end{equation*}
It follows that 
\begin{equation*}
 \begin{split}
  \frac{d}{dt} \sum_{j\neq k}  \|u_{jk}\|_{L^2}^2&\leq C
\sum_{j\neq k} \sum_{l\neq j}\int |u_{jk}||X_{jl}|\frac{1}{|\Psi_{jl}|^2}
\left| \frac{|\Psi_{jl}|^2}{|X_{jl}|^2}-1\right|\,d\sigma\\
 &\leq C\big(\sum_{j\neq k}  \|u_{jk}\|_{L^2}^2\big)^{1/2}
\max_{j\neq k} \left\| \frac{|\Psi_{jk}|^2}{|X_{jk}|^2}-1\right\|_{L^2},
\end{split}
\end{equation*}
and we finally obtain by \eqref{ineq:plus-loin}  
\begin{equation*}
 \left|\frac{d}{dt} \big(\sum_{j,k} \|u_{jk}(t)\|_{L^2}^2\big)^{1/2}\right|
\leq C\mathcal E(t)^{1/2}.
\end{equation*}
The conclusion follows.

\end{proof}

\medskip

We are now able 
to control the evolution of 
$\mathcal E(t)$ and to complete the proof of Theorem \ref{thm:main}.
First we recall that by Proposition \ref{prop:I},
\begin{equation*}
\frac{1}{2}\left(\|v(t)\|_{L^2}^2+\|w(t)\|_{L^2}^2\right)-\tilde{\mathcal{E}_0}\leq
 \mathcal{E}(t)\leq \tilde{\mathcal{E}_0}+\frac{1}{2}\left(\|v(t)\|_{L^2}^2+\|w(t)\|_{L^2}^2\right)
\end{equation*}
so in particular 
\begin{equation*}\begin{split}
 \mathcal{E}(t)+\|v(t)\|_{L^2}^2+\|w(t)\|_{L^2}^2
\leq C\tilde{\mathcal{E}_0}\quad \text{on}\quad  [0,T].
\end{split}\end{equation*}. 

Next, in view of Proposition \ref{prop:est-U} we have
\begin{equation*}\begin{split}\mathcal E(t)
&\leq  \tilde{\mathcal E_0}+(\|v(t)\|_{L^2}+\|w(t)\|_{L^2})^2
\leq \tilde{\mathcal E_0}+2(\|v(0)\|_{L^2}+\|w(0)\|_{L^2})^2\\
&+Ct^2\sup_{s\in [0,t]}\max_{j, k}\|u_{jk}(s)\|_{L^2}
\mathcal E(s)^{1/2}\left(\mathcal{E}(s)^{1/2}+\|v(s)\|_{L^2}+\|w(s)\|_{L^2}\right)^2\\
&\leq 9\tilde{\mathcal{E}_0}+Ct^2\sup_{s\in [0,t]}\max_{j, k}\|u_{jk}(s)\|_{L^2}\tilde{\mathcal E_0}^{3/2}
\end{split}\end{equation*}
and finally by Proposition \ref{prop:est-uj} 
$$\mathcal E(t)\leq 9\tilde{\mathcal E_0}+Ct^2 \max_{j, k}\|u_{jk,0}\|_{L^2}
\tilde{\mathcal E_0}^{3/2}+Ct^3 
\tilde{\mathcal E_0}^2.$$
Setting  $t=T$ in the above inequality and recalling that  $\mathcal{E}(T)=10\tilde{\mathcal{E}_0}$, we infer that
\begin{equation*}
1\leq C
t^2 \max_{j, k}\|u_{jk,0}\|_{L^2}\tilde{\mathcal E_0}^{1/2}+Ct^3\tilde{\mathcal{E}_0}. 
\end{equation*}
We conclude that $T$ is larger than
$$C\min\left\{\frac{1}{\tilde{\mathcal E_0}^{1/4}\max_{j, k} \|u_{jk,0}\|_{L^2}^{1/2}},
\frac{1}{\tilde{\mathcal E_0}^{1/3}}\right\},$$
as we wanted. This concludes the proof of Theorem \ref{thm:main}.

\section{Proof of Theorem \ref{thm:TW}}

\label{sec:tw}

Before proving Theorem \ref{thm:TW} we start with some preliminary computations. 
We mainly follow
 the Appendix of \cite{Gr}. Assume that $v$ is a $\mathcal{C}^\infty$ small energy
solution to Eq. \eqref{eq:TW} such that $v'$ vanishes at infinity. We set $$\eta=1-|v|^2,$$ then $\eta$ vanishes 
at infinity. We decompose $v$ into its real and 
imaginary parts, $v=v_1+iv_2$. Eq. \eqref{eq:TW} gives then the system
$$\left\{\begin{array}{c}\displaystyle-cv_2'+v_1''+\omega\frac{v_1}{v_1^2+v_2^2}-\omega v_1=0,
\\ \displaystyle cv_1'+v_2''+\omega\frac{v_2}{v_1^2+v_2^2}-\omega v_2=0.\end{array}\right.$$
By substracting the first equation multiplied by $v_2$ from the second one multiplied by $v_1$ 
$$(v_1v_2'-v_1'v_2-\frac{c}{2}\eta)'=0,$$
so since $v$ has finite energy we can integrate from infinity and get
\begin{equation}
 \label{eq:TW2}v_1v_2'-v_1'v_2=\frac{c}{2}\eta.\end{equation}
Next we add the first equation multiplied by $v_1'$ to the second one multiplied by $v_2'$, 
$$(v_1'^2+v_2'^2+\omega\ln (v_1^2+v_2^2)-\omega(v_1^2+v_2^2))'=0,$$
so
\begin{equation}
 \label{eq:TW3}
|v'|^2=-\omega\ln(1-\eta)-\omega\eta.
\end{equation}
Finally,  in view of \eqref{eq:TW2} and \eqref{eq:TW3} we can compute
\begin{equation*}
 \begin{split}
\eta''&=-2|v'|^2-2(v_1v_1''+v_2v_2'')\\
&=-2|v'|^2-2v_1(cv_2'-\omega\frac{v_1}{v_1^2+v_2^2}+\omega v_1)-2v_2(-cv_1'-\omega\frac{v_2}{v_1^2+v_2^2}+\omega v_2)\\
&=-2|v'|^2-2c(v_1v_2'-v_1'v_2)+2\omega-2\omega(v_1^2+v_2^2)\\
&=2\omega\ln (1-\eta)+4\omega \eta-c^2\eta. 
 \end{split}
\end{equation*}

So we find
\begin{equation}
 \label{eq:eta}
\eta''-2\omega \ln(1-\eta)+(c^2-4\omega)\eta=0.
\end{equation}
Multiplying by $\eta'$ and integrating we obtain
\begin{equation*}
(\eta')^2+(c^2-4\omega)\eta^2-4\omega \big( (\eta-1)\ln(1-\eta)-\eta\big)=0,
\end{equation*}
which is satisfied if $\eta$ verifies 
\begin{equation}
 \label{eq:eta'}
\eta'=\alpha\Big(-(c^2-4\omega)\eta^2+4\omega\big( (\eta-1)\ln(1-\eta)-\eta\big) \Big)^{1/2},\quad\alpha=\alpha(\sigma)=\pm 1.
\end{equation}

We now turn to the proof of Theorem \ref{thm:TW}.
From now on we look for 
solutions such that $\eta$ is sufficiently small on the whole of $\R$ and for which the right hand side in 
\eqref{eq:eta'} makes sense.  
We introduce
\begin{equation*}
 a(\eta)=-(c^2-4\omega)\eta^2+4\omega\big( (\eta-1)\ln(1-\eta)-\eta\big).
\end{equation*}
For $0<\eta<1$, we perform a 
Taylor expansion for $a$,
\begin{equation*}
 \begin{split}
  a(\eta)
&=(2\omega-c^2)\eta^2-2\omega \frac{\eta^3}{3}-4\omega\sum_{k\geq 4} \frac{\eta^k}{k(k-1)}\end{split}\end{equation*}
therefore\begin{equation*}\begin{split}b(\eta)\equiv\frac{a(\eta)}{\eta^2}=
2\omega-c^2-2\omega \frac{\eta}{3}+r(\eta)
 \end{split}
\end{equation*} with $r(\eta)=o(\eta)\leq 0$ such that $r'(\eta)=O(\eta)$. Let us set
$$\sigma_0=\frac{2\omega-c^2}{\frac{2\omega}{3}}>0,$$
then $b(\sigma_0)\leq 0$. Since on the other hand $b(0)>0$, there exists $\sigma_1\in(0,\sigma_0]$ such that $b(\sigma_1)=0$. Moreover, since for $\eta\in[0,\sigma_0]$ we have 
$b'(\eta)=-\frac{2\omega}{3}+r'(\eta)\leq -\frac{2\omega}{3}+C(2\omega-c^2)<0$ for $2\omega-c^2$ sufficiently small, 
we infer that $b$ is strictly decreasing on $[0,\sigma_0]$ and therefore $\sigma_1$ is the unique zero of $a$ on $]0,\sigma_0]$.

Next, we fix a small parameter $\eps>0$ and we consider the ODE
\begin{equation*}
 \begin{cases}
\displaystyle  y'_\eps(\sigma)=-\sqrt{a(y_\eps(\sigma))},\\
y_\eps(0)=\sigma_1-\eps. 
\end{cases}
\end{equation*}
Since $\sqrt{a}$ is Lipschitz on $[0,x_1-\eps/2)$ we can find a unique maximal solution  on some interval 
$I$ containing the origin. We claim that $\sup I=+\infty.$  We show first
that $0<y_\eps<\sigma_1-\eps$ on $I\cap[0,\infty)$. Indeed, $y_\eps$ is strictly decreasing on $I\cap[0,\infty)$. Assume by contradiction 
that there exists $\overline{\sigma}$ such that $y_\eps(\overline{\sigma})=0$ and $y_\eps>0$ on 
$[0,\overline{\sigma})$. We recall that $b(y)\sim 2\omega-c^2$ when $y\to 0$. Therefore 
\begin{equation*}
 y_\eps'(\sigma)\geq -2\sqrt{2\omega-c^2}y_\eps(\sigma) \quad \text{for } \sigma\in [\overline{\sigma}-\delta,\overline{\sigma}]  
\end{equation*}
with $\delta$ small. Integrating the differential inequality above yields
\begin{equation*}
 y_\eps(\sigma)\geq y_\eps(\overline{\sigma}-\delta)\exp(-2\sqrt{2\omega-c^2}(\sigma-\overline{\sigma}+\delta))\quad \text{on }
[\overline{\sigma}-\delta,\overline{\sigma}],
\end{equation*}
which contradicts the fact that $y_\eps(\overline{\sigma})=0$. Next, since 
$y\mapsto \sqrt{a(y)}$ is Lipschitz and bounded on $[0,\sigma_1-\eps]$ the maximal solution $y_\eps$ 
exists on $[0,\infty)$ which
proves the claim.

We next let $\eps\to 0$. Noting that $y_\eps$ and $y'_\eps$ are uniformly bounded on $[0,\infty)$ we can
pass to the limit to find a solution\footnote{We do not claim that such a solution is unique or maximal.} 
$y$ to the ODE
\begin{equation*}
 \begin{cases}
\displaystyle y'=-\sqrt{a(y)},\quad \sigma\geq 0\\
\displaystyle y(0)=\sigma_1.  
 \end{cases}
\end{equation*}
 We finally set
\begin{equation*}
\eta(\sigma)=y(\sigma)\quad \text{for }\sigma\in [0,+\infty) \quad \text{and } \eta(-\sigma)=\eta(\sigma)=y(\sigma)\quad \text{for }\sigma\in (-\infty,0].
\end{equation*}
Thanks to $\eta(0)=\sigma_1$ and $a(\sigma_1)=0$ we check that $\eta\in C^\infty(\R)$ is a solution of the ODE \eqref{eq:eta}.
Moreover, by the same kind of arguments as before we have $\eta\to 0$, hence $\eta'(\sigma)\sim -\sqrt{2\omega-c^2}\eta(\sigma)$
as $\sigma\to \infty$, which yields the exponential decay $\eta(\sigma)\leq C_\delta 
\eta(0)\exp(-(\sqrt{2\omega-c^2}-\delta)|\sigma|)$ for all $0<\delta<\sqrt{2\omega-c^2}$.

We  complete the proof of Theorem \ref{thm:TW} by looking for a solution of the form
\begin{equation}\label{def:solution-TW}v=\sqrt{1-\eta}\exp(i\theta).\end{equation}
Then according to \eqref{eq:TW2} we must have
\begin{equation}\label{eq:theta'}
 (1-\eta)\theta'=\frac{c\eta}{2}
\end{equation}
(note that in particular $\theta$ is an  increasing function on $\R$).
Therefore for
$
 \theta(\sigma)=\theta_0+\int_0^\sigma\frac{c\eta}{2(1-\eta)}\,d\tau
$
where $\theta_0\in \R$, then 
\begin{equation*}
 |\theta(+\infty)-\theta(-\infty)|\leq \frac{C\eta(0)}{\sqrt{2\omega-c^2}}\leq C\sqrt{2\omega-c^2}.
\end{equation*}
Also, the map defined by \eqref{def:solution-TW} is a solution to \eqref{eq:TW}. It only remains to 
show that $v$ has finite energy.
This clearly holds in view of the exponential decay of $\eta$, of $\eta'$ (by \eqref{eq:eta'}) 
and of $\theta'$ (by \eqref{eq:theta'}) at infinity. Moreover in view of \eqref{eq:theta'} we obtain
\begin{equation*}
 \mathcal{E}(v)\leq C\|\eta\|_{H^1}^2\leq C(2\omega-c^2)^{3/2}
\end{equation*}
and the conclusion of Theorem \ref{thm:TW} follows.

\section{Proof of Theorem \ref{thm:blup}}
\label{sec:blup}
Under the hypothesis of Theorem \ref{thm:blup}, the angular speed of the configuration $(X_j)_j$ is $\omega=0$ so if we set
$$\Psi_j(t,\sigma)=X_j(t)\Phi(t,\sigma)$$
a solution of Syst. \eqref{syst:interaction}, we have shown in Section \S\ref{sec:sym} that $\Phi$ has to solve the linear Schr\"odinger equation,
$$i\dt \Phi+\partial_\sigma^2 \Phi=0.$$
Since the linear evolution of a Gaussian $G_0(\sigma)=e^{-\sigma^2}$ is 
$$e^{it\partial_\sigma^2}G_0(t,\sigma)=\frac{e^{-\frac{\sigma^2}{1+4it}}}{\sqrt{1+4it}},$$
it follows that the linear evolution of 
$$\Phi_0(\sigma)=1-\frac{e^{-\frac{\sigma^2}{1-4i}}}{\sqrt{1-4i}}$$
is precisely
$$\Phi(t,\sigma)=1-\frac{e^{-\frac{\sigma^2}{1-4i(1-t)}}}{\sqrt{1-4i(1-t)}}.$$
We notice that $\Phi(t,\sigma)\overset{|\sigma|\rightarrow\infty}{\longrightarrow} 1$ for $t\in [0,1]$, and for $t\in[0,1[$
$$|\Phi(t,\sigma)|> 1-\frac{1}{\sqrt{1+16(1-t)^2}}>0.$$
On the other hand we have 
$$\Phi(1,\sigma)=1-e^{-\sigma^2},$$
so $\sigma=0$ is a vanishing point at $t=1$ and Theorem  \ref{thm:blup} follows.

\end{document}